\documentclass[12pt,english,sumlimits,reqno,oneside]{amsart}

\usepackage[T1]{fontenc}
\usepackage[utf8]{inputenc}
\usepackage{lmodern}

\usepackage{babel}

\usepackage[a4paper,margin=2.5cm]{geometry}

\usepackage{xifthen}
\usepackage{ifthenx}

\usepackage{xargs}

\usepackage{graphicx}
\usepackage[table,svgnames,x11names]{xcolor}

\usepackage[disable,
colorinlistoftodos,prependcaption,textsize=tiny,textwidth=3.5cm]{todonotes}
\presetkeys{todonotes}{fancyline}{}
\newcommandx{\todoin}[2][1=]{\todo[inline, caption={todo}, #1]{%
    \begin{minipage}{\textwidth-20pt}#2\end{minipage}}}

\newcommandx{\remove}[2][1=]{\todo[linecolor=Plum,backgroundcolor=Plum!25,bordercolor=Plum,#1]{#2}}
\newcommandx{\removein}[2][1=]{\remove[inline, caption={remove}, #1]{%
    \begin{minipage}{\textwidth-20pt}#2\end{minipage}}}

\newcommandx{\remark}[2][1=]{\todo[linecolor=red,backgroundcolor=red!25,bordercolor=red,#1]{#2}}
\newcommandx{\remarkin}[2][1=]{\remark[inline, caption={remark}, #1]{%
    \begin{minipage}{\textwidth-20pt}#2\end{minipage}}}

\newcommandx{\suggestion}[2][1=]{\todo[linecolor=yellow,backgroundcolor=yellow!25,bordercolor=yellow,#1]{#2}}
\newcommandx{\suggestionin}[2][1=]{\suggestion[inline, caption={suggestion}, #1]{%
    \begin{minipage}{\textwidth-20pt}#2\end{minipage}}}

\usepackage{amssymb}
\usepackage{amsmath}
\usepackage{amsthm}
\usepackage{mathtools}

\usepackage{exscale}
\usepackage{relsize}
\usepackage{bbm}
\usepackage{bm}
\usepackage{amsbsy}
\usepackage{mathdots}

\usepackage{mathrsfs}

\usepackage{stmaryrd}

\usepackage{mdframed}

\usepackage[all]{xy}

\usepackage{fp}

\usepackage[section]{placeins}
\usepackage{float}

\usepackage{enumerate}

\newcounter{proof}
{\stepcounter{proof}\begin{proof}}%
{\end{proof}}%
\newcounter{proofstep}[proof]
{\refstepcounter{proofstep}\bigskip\par\noindent%
  \ifthenelse{\isempty{#1}}
    {\textit{Step \theproofstep. }}
    {\textit{#1.}}
  \noindent}%
{\par}%
\newcounter{proofstep-noskip}[proof]
\newenvironment{proofstep-noskip}[1][]%
{\refstepcounter{proofstep-noskip}
  \ifnum\the\value{proofstep-noskip}>1
    \bigskip\par\noindent
  \fi
  \ifthenelse{\isempty{#1}}
    {\textit{Step \theproofstep-noskip. }}
    {\textit{#1.}}
  \noindent}%
{\par}%
\newcounter{proofcase}[proof]
{\refstepcounter{proofcase}\bigskip\par\noindent%
  \ifthenelse{\isempty{#1}}
    {\textit{Case \theproofcase. }}
    {\textit{#1.}}
  \noindent}%
{\par}%

\newcounter{proofcase-noskip}[proof]
\newenvironment{proofcase-noskip}[1][]%
{\refstepcounter{proofcase-noskip}
  \ifnum\the\value{proofcase-noskip}>1
    \bigskip\par\noindent
  \fi
  \ifthenelse{\isempty{#1}}
    {\textit{Case \theproofcase-noskip. }}
    {\textit{#1.}}
  \noindent}%
{\par}%

\usepackage{varioref}
\usepackage{hyperref}
\hypersetup{
  colorlinks,
  allcolors=DarkBlue
}
\usepackage[nameinlink]{cleveref}

\theoremstyle{plain}
\newtheorem{thm}{Theorem}[section]
\newtheorem*{thm*}{Theorem}
\newtheorem{pro}[thm]{Proposition}
\newtheorem{cor}[thm]{Corollary}
\newtheorem{lem}[thm]{Lemma}
\newtheorem{question}[thm]{Question}
\newtheorem{problem}[thm]{Problem}
\theoremstyle{definition}
\newtheorem{dfn}[thm]{Definition}
\theoremstyle{remark}

\newtheorem{rem}[thm]{Remark}

\numberwithin{equation}{section}

\newcommandx{\textref}[2][1=]{\hyperref[#2]{#1\ref*{#2}}}
\newcommandx{\textrefp}[2][1=]{(\hyperref[#2]{#1\ref*{#2}})}

\newcommand{\dif}{\ensuremath{\, \mathrm d}}

\DeclareMathOperator{\spn}{span}

\DeclareMathOperator*{\ave}{ave}

\newcommand{\vvvert}{{\vert\kern-0.25ex\vert\kern-0.25ex\vert}}

\usepackage{tikz-cd}

\begin{document}

\title[Dimension dependence of factorization problems]%
{Dimension dependence of factorization problems: Haar system Hardy spaces}%

\author[T.~Speckhofer]{Thomas Speckhofer}%
\address{Thomas Speckhofer, Institute of Analysis, Johannes Kepler University Linz, Altenberger
  Strasse 69, A-4040 Linz, Austria}%
\email{thomas.speckhofer@jku.at}%
\date{\today}%



\thanks{The author was supported by the Austrian Science Fund~(FWF), project~I5231.}

\begin{abstract}
  For $n\in \mathbb{N}$, let $Y_n$ denote the linear span of the first $n+1$ levels of the Haar system in a \emph{Haar system Hardy space}~$Y$ (this class contains all separable rearrangement-invariant function spaces and also related spaces such as dyadic~$H^1$). Let $I_{Y_n}$ denote the identity operator on $Y_n$.  We prove the following quantitative factorization result: Fix $\Gamma,\delta,\varepsilon > 0$, and let $n,N\in \mathbb{N}$ be chosen such that $N \ge  Cn^2$, where $C = C(\Gamma,\delta,\varepsilon) > 0$ (this amounts to a quasi-polynomial dependence between $\dim Y_N$ and~$\dim Y_n$). Then for every linear operator~$T\colon Y_N\to Y_N$ with $\|T\|\le \Gamma$, there exist operators~$A,B$ with $\|A\|\|B\|\le 2(1+\varepsilon)$ such that either $I_{Y_n} = ATB$ or $I_{Y_n} = A(I_{Y_N} - T)B$. Moreover, if $T$ has $\delta$-large positive diagonal with respect to the Haar system, then we have $I_{Y_n} = ATB$ for some $A,B$ with $\|A\|\|B\|\le (1+\varepsilon)/\delta$. If the Haar system is unconditional in~$Y$, then an inequality of the form $N \ge C n$ is sufficient for the above statements to hold (hence, $\dim Y_N$ depends polynomially on $\dim Y_n$). Finally, we prove an analogous result in the case where $T$ has large but not necessarily positive diagonal entries.
\end{abstract}

\maketitle

\tableofcontents


\section{Introduction and main results}
\label{sec:introduction-and-main-results}

This article deals with quantitative factorization problems in finite-dimensional Banach spaces. We consider problems of the following form:
\begin{problem}\label{problem-1}
Let $X_0\subset X_1\subset X_2\subset \cdots$ be a sequence of finite-dimensional subspaces of a Banach space~$X$, let $(e_j)_{j=1}^{\infty}$ be a sequence in~$X$, and assume that for every~$n\ge 0$, $(e_j)_{j=1}^{d_n}$ is a basis of $X_n$ (where $d_n = \dim X_n$). Given $\Gamma,\delta,\varepsilon > 0$, find conditions on $n,N\in \mathbb{N}_0 = \mathbb{N}\cup \{ 0 \}$ such that the following statement holds: For every linear operator $T\colon X_N\to X_N$ with $\|T\|\le \Gamma$ and with $\delta$-large positive diagonal with respect to $(e_j)_{j=1}^{d_n}$ (i.e., all diagonal entries of the matrix of~$T$ with respect to $(e_j)_{j=1}^{d_n}$ are greater than or equal to~$\delta$), there exist linear operators $B\colon X_n\to X_N$ and $A\colon X_N\to X_n$ with $\|A\|\|B\|\le \frac{1+\varepsilon}{\delta}$ such that the following diagram is commutative:
\begin{equation*}
  \begin{tikzcd}
    X_n \arrow{r}{I_{X_n}} \arrow[swap]{d}{B} & X_n\\
    X_N \arrow{r}{T}& X_N \arrow[swap]{u}{A}
  \end{tikzcd}
\end{equation*}
In that case, we say that \emph{the identity~$I_{X_n}$ factors through~$T$ with constant~$\frac{1+\varepsilon}{\delta}$}.
\end{problem}
We will also consider variations of \Cref{problem-1}, where the condition on the diagonal of~$T$ is removed and the conclusion is replaced by a factorization through~$T$ or $I_{X_N} - T$, or where the diagonal entries only have to be large \emph{in absolute value}.

In the following, we will always take $X_n$ to be the subspace spanned by the first $n+1$ levels of the Haar system in a fixed underlying space~$Y$. Thus, we put $X_n = Y_n := \operatorname{span}\{ h_I : |I|\ge 2^{-n} \}$, where $(h_I)_{I\in \mathcal{D}}$ denotes the Haar system on $[0,1)$, indexed by the dyadic intervals~$I\in \mathcal{D}$ (and we equip every space~$Y_n$ with its Haar basis). In our results, the space $Y$ will be a \emph{Haar system Hardy space}. This class of Banach spaces was introduced in~\cite{LechnerSpeckhofer2023} and~\cite{LechnerMotakisMuellerSchlumprecht2023}, and the precise definition will be given in~\Cref{sec:haar-system-hardy-spaces}. For now, we only mention that $Y$ is the completion of the linear span of the Haar system on $[0,1)$, either under a rearrangement-invariant norm (such as the $L^p$-norm, where $1\le p\le \infty$, or an Orlicz norm) or under an associated norm that can be defined via the square function---hence, the dyadic Hardy space $H^1$, equipped with an equivalent norm, is also contained in this class.
Factorization results for these \emph{infinite-dimensional} spaces have been proved by Kh.~V.~Navoyan~\cite{MR4635015} and by R.~Lechner and the author~\cite{LechnerSpeckhofer2023}.
Moreover, in~\cite{LechnerMotakisMuellerSchlumprecht2023}, R.~Lechner, P.~Motakis, P.F.X.~Müller and Th.~Schlumprecht proved factorization results for Haar multipliers on \emph{bi-parameter} Haar system Hardy spaces.

Note that in the special case where $X_n = Y_n\subset L^p$ for some $1\le p\le \infty$, \Cref{problem-1} can be solved by applying the \emph{Restricted Invertibility Theorem} by J.~Bourgain and L.~Tzafriri~\cite{MR0890420} (see also~\cite{Spielman2009AnEP,MR3726618,MR4425347}): If $R_N$ is the coordinate projection from~$Y_N$ onto $Z_N :=\operatorname{span}\{ h_I : |I| = 2^{-N} \}$, then $R_NT|_{Z_N}$ can be identified with an operator on $\ell^p_{2^N}$ with large positive diagonal (with respect to the unit vector basis). By \cite[Corollary~3.2]{MR0890420}, this operator is \emph{well invertible} on a large subspace of~$\ell^p_{2^N}$ spanned by $2^n \ge c 2^N$ basis elements, where $c>0$ does not depend on~$N$ (see also~\cite[Section~5]{MR2990628} for more details). Since $Y_n$ can be identified with (a subspace of) $\ell^p_{2^n}$, this yields a factorization of~$I_{Y_n}$ through~$T$. Hence, in this case we have linear dimension dependence. Conversely, note that if $I_{Y_n} = ATB$ for some operators $A,B$ with $\|A\|\|B\|\le C$, then $T$ is well invertible on the image of~$B$. Related results on factorization and restricted invertibility of \emph{continuous matrix functions} can be found in~\cite{MR4059948,MR4543386}.

Quantitative bounds for~$N$ (in terms of~$n$) in \Cref{problem-1} have also been proved for many other classical Banach spaces, and by employing \emph{Bourgain's localization method}~\cite{MR0720307}, it is sometimes possible to prove the primarity of a Banach space by reducing this problem to a finite-dimensional factorization problem. We refer to~\cite{MR0955660,MR1045291,MR1419470,MR2280992,MR2302737,MR2990628,MR3436171,MR3794334,MR3910428,MR4229092} for such factorization and primarity results. In spaces with less structure than $L^p$, however, the best known lower bounds on~$N$ are often super-exponential functions of~$n$.

In the case where $Y$ is the dyadic Hardy space $H^p$, $1\le p<\infty$, or $SL^{\infty}$, R.~Lechner~\cite{MR3990955} obtained a factorization result where $N$ depends linearly on $n$, improving on previous results where the estimate for~$N$ was a nested exponential function of~$n$.
Our main result states that in every Haar system Hardy space, an inequality of the form $N\ge C(\Gamma,\delta,\varepsilon)\, n^2$ (where $C(\Gamma,\delta,\varepsilon) > 0$) is sufficient for the conclusion of \Cref{problem-1} and its variation involving~$T$ and $I_{Y_N} - T$ to hold. Moreover, if the Haar system is unconditional in the space under consideration, then $N\ge C(\Gamma,\delta,\varepsilon,K)\, n$ suffices. In the following, the variable~$\delta$ will either denote a lower bound for the diagonal of an operator~$T$, or it will be omitted (i.e., set to~$1$) to state the result about factorizations through $T$ or~$I_{Y_N} - T$.

\begin{thm}\label{thm:main-result}
  Let $Y$ be a Haar system Hardy space, and let $\Gamma,\delta,\varepsilon > 0$. Put $\eta = \frac{\varepsilon}{6(1+\varepsilon)}$. Moreover, let $n, N\in \mathbb{N}_0$ be chosen such that
  \begin{equation}\label{eq:300}
    N \ge  42n(n+1)\biggl\lceil \frac{\Gamma}{\eta\delta} \biggr\rceil + 42 + \Bigl\lfloor 4\log_2\Bigl(\frac{\Gamma}{\eta\delta}\Bigr) \Bigr\rfloor.
  \end{equation}
  Then for every linear operator $T\colon Y_N\to Y_N$ with $\|T\|\le \Gamma$, the following holds:
  \begin{enumerate}[(i)]
    \item\label{thm:main-result:i} If $\delta = 1$ in~\eqref{eq:300}, then the identity $I_{Y_n}$ factors through $T$ or $I_{Y_N} - T$ with constant $2(1+\varepsilon)$.
    \item\label{thm:main-result:ii} If $T$ has $\delta$-large positive diagonal with respect to the Haar basis, then the identity~$I_{Y_n}$ factors through~$T$ with constant~$\frac{1+\varepsilon}{\delta}$.
  \end{enumerate}
  If the Haar system is $K$-unconditional in~$Y$ for some $K\ge 1$, then inequality~\eqref{eq:300} can be replaced by
  \begin{equation}\label{eq:302}
    N \ge 42n\biggl\lceil \frac{K\Gamma}{\eta\delta} \biggr\rceil + 42 + \Bigl\lfloor 4\log_2\Bigl(\frac{\Gamma}{\eta\delta}\Bigr) \Bigr\rfloor.
  \end{equation}
\end{thm}
In the following, we will denote the right-hand side of~\eqref{eq:300} by $N_{\min}(\Gamma,\delta,\varepsilon,n)$ and the right-hand side of~\eqref{eq:302} by $N_{\min}^K(\Gamma,\delta,\varepsilon,n)$.

\begin{rem}
  Due to inequality~\eqref{eq:300}, we have a quasi-polynomial dependence between $\dim Y_N\sim 2^N$ and~$\dim Y_n\sim 2^n$ in the general case, whereas by~\eqref{eq:302}, we have polynomial dimension dependence in the unconditional case.
\end{rem}

We will prove \Cref{thm:main-result} in \Cref{sec:stabilization}. The basic structure of our proof is similar to that of the proofs in~\cite{MR3990955} and~\cite{LechnerSpeckhofer2023}: Starting with a given operator~$T$, we perform a step-by-step reduction which yields ``simpler'' operators. First, we use a modified version of the probabilistic argument in~\cite{MR3990955} to \emph{diagonalize} the operator~$T$ (up to a small error). Next, we \emph{stabilize} the resulting diagonal operator~$D$ by using a discrete version of the stabilization procedures developed in~\cite{MR4430957,LechnerSpeckhofer2023,LechnerMotakisMuellerSchlumprecht2023} (here, we just use combinatorial arguments instead of cluster points and randomization). Finally, by a perturbation argument, we obtain a constant multiple of the identity operator~$I_{Y_n}$.
In every reduction step, we obtain a factorization (or an approximate factorization) of the newly constructed operator through the previous one, and by combining all these factorizations, we obtain the desired results.

Observe that if the Haar system is $K$-unconditional in~$Y$ for some $K\ge 1$, then \Cref{thm:main-result}~\eqref{thm:main-result:ii} also holds for operators~$T$ for which the \emph{absolute values} of all diagonal entries are at least $\delta$ (we just have to increase the factorization constant by a factor~$K$).
However, by performing an additional reduction step, we also can drop the assumption that~$T$ has positive diagonal entries in the general case. We have the following result:
\begin{cor}\label{cor:large-diagonal}
  Let $Y$ be a Haar system Hardy space, and let $\Gamma,\delta,\varepsilon > 0$. Moreover, let $n,\widetilde{N}\in \mathbb{N}_0$ be chosen such that
  \begin{equation}\label{eq:301}
    \widetilde{N} \ge 2N \bigg\lceil \frac{N}{\varepsilon} + 1 \bigg\rceil 2^{N},
  \end{equation}
  where $N = N_{\min}(2(1+\varepsilon)\Gamma,\delta,\varepsilon,n)$, or where $N = N_{\min}^K(2(1+\varepsilon)\Gamma,\delta,\varepsilon,n)$ if the Haar system is $K$-unconditional in~$Y$. Then for every linear operator $\widetilde{T}\colon Y_{\widetilde{N}}\to Y_{\widetilde{N}}$ with $\|\widetilde{T}\|\le \Gamma$ and with $\delta$-large diagonal with respect to the Haar basis, the identity~$I_{Y_n}$ factors through~$\widetilde{T}$ with constant~$\frac{2(1+\varepsilon)^2}{\delta}$.
\end{cor}

We will prove \Cref{cor:large-diagonal} in \Cref{sec:reduction-to-positive-diagonal}.
Note that by~\eqref{eq:301}, $\dim Y_{\widetilde{N}}$ is a double exponential function of~$(\log\dim Y_n)^2$, whereas in the unconditional case, $\dim Y_{\widetilde{N}}$ depends exponentially on~$\dim Y_n$.

\section{Preliminaries}\label{sec:preliminaries}

\subsection{Notation and basic definitions}
\label{sec:notation}

Let $\mathbb{N}_0 = \mathbb{N}\cup \{ 0 \}$, and let $\mathcal{D}$ denote the collection of all dyadic intervals in~$[0,1)$, i.e.,
\begin{equation*}
  \mathcal{D}
  = \Big\{\Big[\frac{i-1}{2^j},\frac{i}{2^j}\Big) :
  j\in\mathbb{N}_0,\ 1\leq i\leq 2^j\Big\}.
\end{equation*}
In addition, for each $n\in\mathbb{N}_0$, define
\begin{equation*}
  \mathcal{D}_n
  = \{I\in\mathcal{D} : |I| = 2^{-n}\}
  \qquad
  \text{and}
  \qquad
  \mathcal{D}_{\le n}
  = \bigcup_{k=0}^n \mathcal{D}_k,
  \quad
  \mathcal{D}_{<n} = \bigcup_{k=0}^{n-1}\mathcal{D}_k.
\end{equation*}
If $I\in \mathcal{D}$ is a dyadic interval, then we denote by~$I^+$ the left half of~$I$ and by~$I^-$ the right
right half of~$I$ (both are elements of~$\mathcal{D}$). If we use the symbol~$\pm$ multiple times in
an equation, we mean either always~$+$ or always~$-$. For any
subcollection $\mathcal{B}\subset\mathcal{D}$, we put $\mathcal{B}^* = \bigcup_{I\in \mathcal{B}}I$.

Next, we define the bijective function $\iota\colon\mathcal{D}\to\mathbb{N}$ by
\begin{equation*}
  \Big[\frac{i}{2^j},\frac{i+1}{2^j}\Big)
  \overset{\iota}{\mapsto} 2^j + i.
\end{equation*}
The Haar system $(h_I)_{I\in\mathcal{D}}$ is defined as
\begin{equation*}
  h_I
  = \chi_{I^+} - \chi_{I^-},
  \qquad I\in\mathcal{D},
\end{equation*}
where $\chi_A$ denotes the characteristic function of a subset $A\subset [0,1)$.  We additionally
put $h_\varnothing = \chi_{[0,1)}$ and $\mathcal{D}^+ = \mathcal{D}\cup\{\varnothing\}$ as well as
$\iota(\varnothing) = 0$.  Recall that in the linear order induced by $\iota$, the Haar system
$(h_I)_{I\in \mathcal{D}^+}$, is a monotone Schauder basis
of $L^p$, $1\le p<\infty$ (and unconditional if $1<p<\infty$). For
$x = \sum_{I\in \mathcal{D}^+} a_Ih_I\in L^1$, by the \emph{Haar support} of~$x$ we mean the
set of all $I\in \mathcal{D}^+$ such that $a_I\ne 0$. We only consider Banach spaces over the real numbers.
If $\mathcal{B}\subset \mathcal{D}^+$ is a non-empty collection of dyadic intervals with cardinality~$\#\mathcal{B} <\infty$ and
$f\colon \mathcal{B} \to \mathbb{R}$ is a function, then we use the notation
\begin{equation*}
  \ave_{K\in \mathcal{B}} f(K) = \frac{1}{\#\mathcal{B}} \sum_{K\in \mathcal{B}} f(K).
\end{equation*}
Finally, we will occasionally use the symbols $\gtrsim$, $\lesssim$, $\sim$ to suppress constants in inequalities: $a\gtrsim b$ and $b\lesssim a$ both mean that $a\ge Cb$ for an absolute constant $C > 0$, and $a\sim b$ means that $a\gtrsim b$ and $b\gtrsim a$.

\subsection{Haar system Hardy spaces}
\label{sec:haar-system-hardy-spaces}

The class of Haar system Hardy spaces was introduced in~\cite{LechnerSpeckhofer2023,LechnerMotakisMuellerSchlumprecht2023}. To give the definition, we first need the notion of a
\emph{Haar system space} (see~\cite[Section~2.4]{MR4430957}).
\begin{dfn}\label{dfn:HS-1d}
  A \emph{Haar system space $X$} is the completion of
  $H := \spn\{ h_I : I\in\mathcal{D}^+\} = \spn\{\chi_I : I\in\mathcal{D}\}$ under a norm
  $\|\cdot\|_X$ that satisfies the following properties:
  \begin{enumerate}[(i)]
    \item\label{dfn:HS-1d:1} If $f$, $g$ are in $H$ and $|f|$, $|g|$ have the same distribution,
          then $\|f\|_X = \|g\|_X$.
    \item\label{dfn:HS-1d:2} $\|\chi_{[0,1)}\|_X = 1$.
  \end{enumerate}
  We denote the class of Haar system spaces by $\mathcal{H}(\delta)$.  Given
  $X\in\mathcal{H}(\delta)$, one can define the closed subspace $X_0$ of $X$ as the closure of
  $H_0 = \spn\{ h_I : I\in \mathcal{D} \}$ in $X$.  We denote the class of these subspaces by
  $\mathcal{H}_0(\delta)$.
\end{dfn}
Apart from the spaces $L^p$, $1\le p<\infty$, and the closure of $H$ in $L^{\infty}$, the class
$\mathcal{H}(\delta)$ includes all rearrangement-invariant function spaces~$X$ on $[0,1)$ (e.g.,
Orlicz function spaces) in which the linear span~$H$ of the Haar system $(h_I)_{I\in \mathcal{D}^+}$ is
dense. If $H$ is not dense in $X$, then its closure in $X$ is a Haar system space.

The following proposition contains some basic results on Haar system spaces which will be used frequently throughout the paper. For proofs we refer to~\cite[Section~2.4]{MR4430957} and~\cite{LechnerSpeckhofer2023}.

\begin{pro}\label{pro:HS-1d}
  Let $X\in\mathcal{H}(\delta)$ be a Haar system space. Then the following holds:
  \begin{enumerate}[(i)]
    \item\label{pro:HS-1d:i} For every $f\in H = \spn\{h_I\}_{\in\mathcal{D}^+}$, we have
          $\|f\|_{L^1}\leq \|f\|_X\leq \|f\|_{L^\infty}$.
    \item\label{pro:HS-1d:v} For all $f,g\in H$ with $|f|\leq |g|$, we have $\|f\|_X \leq \|g\|_X$.
    \item\label{pro:HS-1d:ii} The Haar system $(h_I)_{I\in \mathcal{D}^+}$, in the usual linear
          order, is a monotone Schauder basis of~$X$.
    \item \label{pro:HS-1d:iii} $H$ naturally coincides with a subspace of $X^{*}$ (where $f\in H$ acts on $g\in H$ by $\langle f,g \rangle = \int fg$), and its closure in $X^{*}$ is also a Haar system space.
    \item For every $I\in \mathcal{D}$, we have $\|h_I\|_X\|h_I\|_{X^{*}} = |I|$.
  \end{enumerate}
\end{pro}

Next, by $(r_n)_{n=0}^{\infty}$, we denote the sequence of standard Rademacher functions, i.e.,
\begin{equation*}
  r_n = \sum_{I\in \mathcal{D}_n} h_I,\qquad n\in \mathbb{N}_0.
\end{equation*}
We define the set
\begin{equation*}
  \mathcal{R} = \{ (r_{\iota(I)})_{I\in \mathcal{D}^+}, (r_0)_{I\in \mathcal{D}^+} \}.
\end{equation*}
Hence, if $\mathbf{r} = (r_I)_{I\in \mathcal{D}^+}\in \mathcal{R}$, then $\mathbf{r}$ is either an
independent sequence of $\pm 1$-valued random variables or a constant
sequence.

\begin{dfn}\label{dfn:haar-system-hardy-space}
  Given $X\in\mathcal{H}(\delta)$ and
  $\mathbf{r} = (r_I)_{I\in\mathcal{D}^+}\in\mathcal{R}$, we define the \emph{(one-parameter) Haar
    system Hardy space $X(\mathbf{r})$} as the completion of
  $H = \spn\{ h_I : I\in \mathcal{D}^+ \}$ under the norm $\|\cdot\|_{X(\mathbf{r})}$ given by
  \begin{equation*}
    \Bigl\| \sum_{I\in\mathcal{D}^+} a_I h_I\Bigr\|_{X(\mathbf{r})}
    = \Bigl\|
    s\mapsto \int_0^1 \Bigl|
    \sum_{I\in\mathcal{D}^+} r_I(u) a_I h_I(s)
    \Bigr| \dif u
    \Bigr\|_X.
  \end{equation*}
  We denote the class of one-parameter Haar system Hardy spaces by $\mathcal{HH}(\delta)$.
  Moreover, given $X(\mathbf{r})\in\mathcal{HH}(\delta)$, we define the subspace
  $X_0(\mathbf{r})$ as the closure of $H_0 = \operatorname{span}\{ h_I \}_{I\in \mathcal{D}}$ in~$X(\mathbf{r})$.
  For notational convenience, we will also refer to the subspaces $X_0(\mathbf{r})$ as Haar system Hardy spaces.
  We denote the class of these subspaces by $\mathcal{H}\mathcal{H}_0(\delta)$.
\end{dfn}
Note that if $r_I = r_0$ for all $I\in\mathcal{D}^+$, then we have $\|\cdot \|_{X(\mathbf{r})} = \|\cdot \|_X$ and thus $X(\mathbf{r}) = X$, so $\mathcal{HH}(\delta)$ is an extension of $\mathcal{H}(\delta)$. Moreover, if $\mathbf{r}$ is independent, then by Khintchine's inequality, the norm on $X(\mathbf{r})$ can be expressed in terms of the square function:
\begin{equation*}
  \Bigl\| \sum_{I\in \mathcal{D}^+}a_Ih_I \Bigr\|_{X(\mathbf{r})} \sim \Bigl\| \Bigl( \sum_{I\in \mathcal{D}^+} a_I^2h_I^2 \Bigr)^{1/2} \Bigr\|_X.
\end{equation*}
Thus, for $X = L^1$ and an independent Rademacher sequence~$\mathbf{r}$, we obtain the dyadic Hardy space~$H^1$, equipped with an equivalent norm.

In the next proposition, we collect some basic properties of Haar system Hardy spaces. They are either proved in~\cite{LechnerSpeckhofer2023}, or they follow by a simple direct argument or calculation.

\begin{pro}\label{pro:HSHS-1d}
  Let $X(\mathbf{r})\in\mathcal{HH}(\delta)$, and put $Y = X_0(\mathbf{r})$. Then following holds:
  \begin{enumerate}[(i)]
    \item The Haar system $(h_I)_{I\in \mathcal{D}}$, in the usual linear
          order, is a monotone Schauder basis of~$Y$, and it is $1$-unconditional if $\mathbf{r}$ is independent.
    \item $H_0 = \operatorname{span}\{ h_I \}_{I\in \mathcal{D}}$ naturally coincides with a subspace of $Y^{*}$, and if $\mathbf{r}$ is independent, then $(h_I)_{I\in \mathcal{D}}$ is $1$-unconditional in $Y^{*}$.
    \item\label{pro:HSHS-1d:iii} For every $I\in \mathcal{D}$, we have $\|h_I\|_Y\|h_I\|_{Y^{*}} = |I|$. In particular, we have $\|h_I\|_Y = \|h_I\|_X$ and $\|h_I\|_{Y^{*}} = \|h_I\|_{X^{*}}$.
    \item If $f$ is a finite linear combination of disjointly supported Haar functions, then we have $\|f\|_{X(\mathbf{r})} = \|f\|_X$.
    \item\label{pro:HSHS-1d:single-layer} If $x = \sum_{I\in \mathcal{D}} a_Ih_I\in H_0$, then for every $k\in \mathbb{N}_0$, we have $\|\sum_{I\in \mathcal{D}_k}a_Ih_I\|_Y\le \|x\|_Y$. Thus, $\|a_Ih_I\|_Y\le \|x\|_Y$ for all $I\in \mathcal{D}$.
    \item\label{pro:HSHS-1d:v} If $\mathcal{B}\subset \mathcal{D}$ is a finite collection of pairwise disjoint dyadic intervals, then $(h_K)_{K\in \mathcal{B}}$ is $1$-unconditional in~$X(\mathbf{r})^{*}$, and we have $\|\sum_{K\in \mathcal{B}}h_K\|_{X(\mathbf{r})^{*}}\le 1$.
  \end{enumerate}
\end{pro}

In the following, $Y$ will always denote a Haar system Hardy space $X_0(\mathbf{r})\in \mathcal{HH}_0(\delta)$, and we will consider operators defined on the subspaces $Y_n = \operatorname{span}\{ h_I \}_{I\in \mathcal{D}_{\le n}}\subset Y$, $n\in \mathbb{N}_0$.

\subsection{Faithful Haar systems}

Next, we introduce a finite version of the concept of a \emph{faithful Haar system}. A faithful Haar system is a system of functions which are blocks of the Haar system and share many structural properties with the original Haar system. For example, the supports of these functions exhibit the same intersection pattern as the supports of the original Haar functions.
The term \emph{faithful Haar system} was coined in~\cite{MR4430957}, but the construction can already be found, e.g., in~\cite[p.~51]{MR0402915}.
We will also use the more general notion of an \emph{almost faithful Haar system}, which allows for gaps between the supports of the functions (cf.~\cite{LechnerSpeckhofer2023}).
These types of constructions originated in classical works such as the paper~\cite{MR0328575} by J.~L.~B.~Gamlen and R.~J.~Gaudet. For more information on these and related concepts, such as \emph{Jones' compatibility conditions}~\cite{MR0796906}, we refer to~\cite{MR2157745}. A brief overview can also be found in~\cite[Section~5]{LechnerSpeckhofer2023}.
\begin{dfn}\label{dfn:faithful}
  Let $n\in \mathbb{N}_0$. For every $I\in \mathcal{D}_{\le n}$, let $\mathcal{B}_I$ be a non-empty finite subcollection of $\mathcal{D}$, and moreover, let $(\theta_K)_{K\in \mathcal{D}}$ be a family of signs. Put $\tilde{h}_I = \sum_{K\in \mathcal{B}_I} \theta_Kh_K$ for $I\in \mathcal{D}_{\le n}$. We say that $(\tilde{h}_I)_{I\in \mathcal{D}_{\le n}}$ is a (finite) \emph{almost faithful Haar system} if the following conditions are satisfied:
\begin{enumerate}[(i)]
  \item\label{dfn:faithful:i} For every $I\in \mathcal{D}_{\le n}$, the collection $\mathcal{B}_I$ consists of pairwise disjoint dyadic intervals, and we have $\mathcal{B}_I\cap \mathcal{B}_J = \emptyset$ for all $I\ne J\in \mathcal{D}_{\le n}$.
  \item\label{dfn:faithful:ii} For every $I\in \mathcal{D}_{<n}$, we have $\mathcal{B}_{I^{\pm }}\subset \{ \tilde{h}_I = \pm 1 \}$.
\end{enumerate}
If $\mathcal{B}_{[0,1)}^{*} = [0,1)$ and $\mathcal{B}_{I^{\pm }} = \{ \tilde{h}_I = \pm 1 \}$ for all $I\in \mathcal{D}_{<n}$, then we say that $(\tilde{h}_I)_{I\in \mathcal{D}_{\le n}}$ is \emph{faithful}, and in this case, we usually denote the system by $(\hat{h}_I)_{I\in \mathcal{D}_{\le n}}$. If $k_0<k_1<\dots<k_n$ is a strictly increasing sequence of integers such that $\mathcal{B}_I\subset \mathcal{D}_{k_i}$ for all $I\in \mathcal{D}_i$, $0\le i\le n$, then we say that $(\tilde{h}_I)_{I\in \mathcal{D}_{\le n}}$ is an almost faithful Haar system \emph{with frequencies $k_0,\dots,k_n$} (note that in the terminology of~\cite{LechnerSpeckhofer2023}, the frequencies would be denoted as $(k_I)_{I\in \mathcal{D}_{\le n}}$, where $k_I = k_i$ whenever $|I| = 2^{-i}$). An example of a (finite) faithful Haar system is depicted in \Cref{fig:faithful-haar-system}.
\end{dfn}

\begin{figure}[h]
  \centering \includegraphics{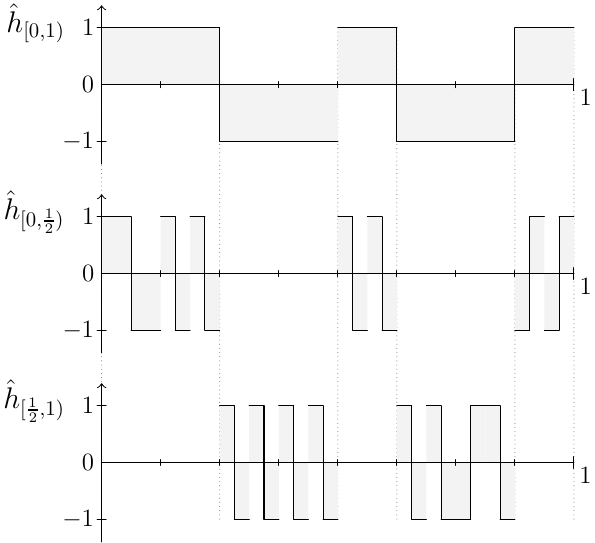}
  \caption{A finite faithful Haar system}
  \label{fig:faithful-haar-system}
\end{figure}

If $(\hat{h}_I)_{I\in \mathcal{D}_{\le n}}$ is a faithful Haar system and $N\in \mathbb{N}_0$ is chosen such that $\mathcal{B}_I\subset \mathcal{D}_{\le N}$ for all $I\in \mathcal{D}_{\le n}$, then we can always define two associated operators:
\begin{pro}\label{pro:operators-Ahat-Bhat}
  Let $Y\in \mathcal{HH}_0(\delta)$ be a Haar system Hardy space. Fix $n,N\in \mathbb{N}_0$, and let $(\hat{h}_I)_{I\in \mathcal{D}_{\le n}}$ be a (finite) faithful Haar system in $Y_N$. Define the associated operators $\hat{B}\colon Y_n\to Y_N$ and $\hat{A}\colon Y_N\to Y_n$ by
  \begin{equation*}
    \hat{B}x = \sum_{I\in \mathcal{D}_{\le n}} \frac{\langle h_I, x \rangle}{|I|} \hat{h}_I
    \qquad \text{and}\qquad
    \hat{A}y = \sum_{I\in \mathcal{D}_{\le n}} \frac{\langle \hat{h}_I, y \rangle}{|I|} h_I
  \end{equation*}
  for $x\in Y_n$ and $y\in Y_N$. Then we have $\hat{A}\hat{B} = I_{Y_n}$ and $\|\hat{A}\| = \|\hat{B}\| = 1$.
\end{pro}
\begin{proof}
  This result follows from \cite[Proposition~7.1]{LechnerSpeckhofer2023} by suitably extending $(\hat{h}_I)_{I\in \mathcal{D}_{\le n}}$ to an \emph{infinite} faithful Haar system $(\hat{h}_I)_{I\in \mathcal{D}}$ (see \cite[Definition~5.1]{LechnerSpeckhofer2023}) and restricting the associated operators $\hat{B}$ and $\hat{A}$ to $Y_n$ and $Y_N$, respectively.
\end{proof}

In order to prove an analogous statement for \emph{almost} faithful Haar systems, we will need an upper bound for the norm of a Haar multiplier, i.e., a linear operator $M\colon Y_n\to Y_n$ with $M h_I = m_I h_I$ for all $I\in \mathcal{D}_{\le n}$, where $(m_I)_{I\in \mathcal{D}_{\le n}}$ is a family of scalars (the \emph{entries} of~$M$). In our finite-dimensional setting, the following inequality yields a better result than the bound in \cite[Lemma~4.5]{LechnerSpeckhofer2023}.
\begin{lem}\label{lem:multiplier-lemma}
  Let $Y$ be a Haar system Hardy space, and let $n\in \mathbb{N}_0$. Then for every Haar multiplier $M\colon Y_n\to Y_n$ with entries $(m_I)_{I\in \mathcal{D}_{\le n}}$, we have
  \begin{equation*}
    \|M - m_{[0,1)}I_{Y_n}\| \le \sum_{k=1}^n \max_{I\in \mathcal{D}_k} |m_I - m_{[0,1)}|.
  \end{equation*}
\end{lem}
\begin{proof}
  Let $x = \sum_{I\in \mathcal{D}_{\le n}} a_Ih_I\in Y_n$. Then we have
  \begin{align*}
    \|Mx - m_{[0,1)}x\|_Y &= \Big\| \sum_{k=1}^n\sum_{I\in \mathcal{D}_k} (m_I - m_{[0,1)})a_Ih_I \Big\|_Y\\
    &\le \sum_{k=1}^n \max_{I\in \mathcal{D}_k} |m_I - m_{[0,1)}| \Big\| \sum_{I\in \mathcal{D}_k} a_Ih_I \Big\|_Y,
  \end{align*}
  and the conclusion follows since by \Cref{pro:HSHS-1d}~\eqref{pro:HSHS-1d:single-layer}, we have $\|\sum_{I\in \mathcal{D}_k} a_Ih_I\|_Y \le \|x\|_Y$ for all~$k$.
\end{proof}

\begin{pro}\label{pro:operators-A-B-almost-faithful}
  Let $Y$ be a Haar system Hardy space, and let $n,N\in \mathbb{N}_0$. Moreover, let $(\tilde{h}_I)_{I\in \mathcal{D}_{\le n}}$ be an almost faithful Haar system in~$Y_N$, and for each $I\in \mathcal{D}_{\le n}$, let $\mathcal{B}_I$ denote the Haar support of~$\tilde{h}_I$. Let $B\colon Y_n\to Y_N$ and $A\colon Y_N\to Y_n$  be the associated operators defined by
  \begin{equation*}
    Bx = \sum_{I\in \mathcal{D}_{\le n}} \frac{\langle h_I, x \rangle}{|I|} \tilde{h}_I\qquad
    \text{and}\qquad
    Ay = \sum_{I\in \mathcal{D}_{\le N}} \frac{\langle \tilde{h}_I, y \rangle}{|\mathcal{B}_I^{*}|} h_I
  \end{equation*}
  for $x\in Y_n$ and $y\in Y_N$. Put $\mu = |\mathcal{B}_{[0,1)}^{*}|$. Then we have $AB = I_{Y_n}$ as well as
  \begin{equation*}
    \|B\| \le 1\qquad
    \text{and}\qquad
    \|A\| \le \frac{1}{\mu} + \sum_{k=1}^n \max_{I\in \mathcal{D}_k} \bigg( \frac{|I|}{|\mathcal{B}_I^{*}|} - \frac{1}{\mu} \bigg).
  \end{equation*}
\end{pro}
\begin{proof}
  Let $Q\colon Y_N\to Y_n$ be defined by
  \begin{equation*}
    Qy = \sum_{I\in \mathcal{D}_{\le N}} \frac{\langle \tilde{h}_I, y \rangle}{|I|} h_I,\qquad  y\in Y_N.
  \end{equation*}
  Like in the proof of \cite[Theorem~7.3]{LechnerSpeckhofer2023}, we see that the linear operators $B$ and~$Q$ satisfy $\|B\|\le 1$ and $\|Q\|\le 1$. Note that $A = MQ$, where $M\colon Y_n\to Y_n$ is the Haar multiplier defined by
  \begin{equation*}
    Mh_I = \frac{|I|}{|\mathcal{B}_I^{*}|}h_I,\qquad I\in \mathcal{D}_{\le n}.
  \end{equation*}
  Thus, the conclusion follows from \Cref{lem:multiplier-lemma}.
\end{proof}

\subsection{Factorization of operators}

Finally, we introduce some convenient terminology for stating factorization results and describing properties of linear operators and their diagonals (cf.~\cite{MR4430957,LechnerSpeckhofer2023}).
\begin{dfn}\label{dfn:factorization-modes}
  Let $X$ and $Y$ denote Banach spaces. Let $S\colon X\to X$ and $T\colon Y\to Y$ be bounded linear operators, and let
  $C, \eta \ge 0$.
  \begin{enumerate}[(i)]
    \item\label{enu:dfn:factorization-modes:a} We say that \emph{$S$ factors through $T$ with
          constant $C$ and error $\eta$} if there exist bounded linear operators $B\colon X\to Y$ and $A\colon Y\to X$ with
          $\|A\|\|B\|\leq C$ such that $\|S - ATB\|\leq \eta$.
    \item\label{enu:dfn:factorization-modes:a2} If~\eqref{enu:dfn:factorization-modes:a} holds and
          we additionally have $AB = I_X$, then we say that $S$ \emph{projectionally} factors
          through $T$ with constant $C$ and error $\eta$.
  \end{enumerate}
  If we omit the phrase ``with error $\eta$'' in~\eqref{enu:dfn:factorization-modes:a}
  or~\eqref{enu:dfn:factorization-modes:a2}, then we take that to mean that the error is~$0$.
\end{dfn}

The following observation will play an important role in the proof of \Cref{thm:main-result}~\eqref{thm:main-result:i}: Let $X,Y$ and $S,T$ be as above and suppose that $S$ \emph{projectionally} factors through~$T$ with constant~$C$ and error~$\eta$.  Then $I_X - S$ projectionally factors through $I_Y - T$ with the same constant and error (cf.~\cite[Remark~2.2]{MR4430957}).

\begin{dfn}\label{dfn:factorization-properties}
  Let $Y$ be a Haar sytem Hardy space, and for $n\in \mathbb{N}_0$, let $Y_n = \operatorname{span}\{ h_I \}_{I\in \mathcal{D}_{\le n}}\subset Y$.
  Moreover, let $\delta > 0$, and let $T\colon Y_n\to Y_n$ be a linear operator. We say that
          \begin{itemize}
            \item $T$ is a \emph{Haar multiplier} (or \emph{diagonal with respect to the Haar system}) if we have
                  $\langle h_I, Th_J \rangle = 0$ for all $I,J\in \mathcal{D}_{\le n}$ with $I\ne J$.
            \item $T$ has \emph{$\delta$-large diagonal (with respect to the Haar system)} if
                  $|\langle h_I, Th_I \rangle| \ge \delta |I|$ for all $I\in \mathcal{D}_{\le n}$.
            \item $T$ has $\delta$-large \emph{positive} diagonal if
                  $\langle h_I, Th_I \rangle \ge \delta |I|$ for all $I\in \mathcal{D}_{\le n}$.
          \end{itemize}
\end{dfn}



\section{Diagonalization via random faithful Haar systems}

As a first step towards proving our main results, we reduce an arbitrary linear operator $T\colon Y_N\to Y_N$ to a Haar multiplier $D\colon Y_n\to Y_n$, where $N$ depends linearly on $n$. This is achieved by constructing a randomized faithful Haar system $(\hat{h}_I)_{I\in \mathcal{D}_{\le n}}$ in $Y_N$ and using the associated operators $\hat{A},\hat{B}$ as defined in~\Cref{pro:operators-Ahat-Bhat}: By a probabilistic argument based on~\cite{MR3990955}, we show that for at least one realization of our random system, we have an approximate factorization $D = \hat{A}T\hat{B} + \Delta$, where $D$ is a Haar multiplier and $\Delta$ is a small error. Moreover, if~$T$ has $\delta$-large positive diagonal with respect to the Haar system for some $\delta > 0$, then the diagonalization preserves this property (the new diagonal entries will only be slightly smaller).

In~\cite{MR3990955}, R.~Lechner used the following randomized construction to prove a diagonalization result in the case where $Y$ is the dyadic Hardy space $H^p$, $1\le p< \infty$, or $SL^{\infty}$: The first step is to select pairwise disjoint finite collections $\mathcal{B}_I\subset \mathcal{D}_{\le N}$, which determine a faithful Haar system
\begin{equation*}
  b_I = \sum_{K\in \mathcal{B}_I} h_K,\qquad I\in \mathcal{D}_{\le n}.
\end{equation*}
Next, random signs $\theta = (\theta_K)_{K\in \mathcal{D}_{\le N}}$ are introduced to obtain the functions
\begin{equation*}
  b_I^{(\theta)} = \sum_{K\in \mathcal{B}_I} \theta_Kh_K,\qquad I\in \mathcal{D}_{\le n},\ \theta\in \{ \pm 1 \}^{\mathcal{D}_{\le N}}.
\end{equation*}
Note that in general, $(b_I^{(\theta)})_{I\in \mathcal{D}_{\le n}}$ is \emph{not} a faithful Haar system since the newly introduced signs $\theta_K$ can break property~\eqref{dfn:faithful:ii} of \Cref{dfn:faithful}. At this point, the proof in~\cite{MR3990955} exploits the fact that the Haar system $(h_K)_{K\in \mathcal{D}}$ is $1$-unconditional in~$H^p$ and~$SL^{\infty}$: For each~$\theta$, the function~$b_I^{(\theta)}$ is obtained by applying a fixed Haar multiplier of norm~$1$ to $b_I$. Thus, the associated operators $\hat{A}$ and $\hat{B}$, defined like in \Cref{pro:operators-Ahat-Bhat} with respect to $(b_I^{(\theta)})_{I\in \mathcal{D}_{\le n}}$, are also bounded, and for at least one choice of~$\theta$, they yield the desired approximate factorization $D = \hat{A}T \hat{B} + \Delta$.

In our case, however, $Y$ is an arbitrary Haar system Hardy space, so in general we cannot use unconditionality to complete the proof as described above. However, a modified randomization procedure, which always yields a \emph{faithful} Haar system, enables us to prove the diagonalization result. Finally, it is worth noting that our resulting Haar multiplier $D$ is already stable along every level of the Haar system, i.e., if $I,J\in \mathcal{D}_{\le n}$ satisfy $|I| = |J|$, then we have $d_I\approx d_J$ (we will utilize this property later).

\begin{pro}\label{pro:diagonalization}
  Let $Y = X_0(\mathbf{r})$ be a Haar system Hardy space, and let $\Gamma,\delta,\eta > 0$. Moreover, let $n,N\in \mathbb{N}_0$ be chosen such that
  \begin{equation*}
    N\ge 21(n+1) + \Bigl\lfloor 4\log_2\Bigl(\frac{\Gamma}{\eta\delta}\Bigr) \Bigr\rfloor.
  \end{equation*}
  Then for every linear operator $T\colon Y_N\to Y_N$ with $\|T\|\le \Gamma$, there exists a Haar multiplier $D\colon Y_n\to Y_n$ with entries $(d_I)_{I\in \mathcal{D}_{\le n}}$ such that $D$ projectionally factors through~$T$ with constant~$1$ and error~$\eta\delta$. Moreover, the following holds:
  \begin{enumerate}[(i)]
    \item\label{pro:diagonalization:item:1} We have $|d_I|\le \|T\|$ for all $I\in \mathcal{D}_{\le n}$.
    \item\label{pro:diagonalization:item:2} We have $|d_I - d_J| \le 8^{-n}\eta\delta$ whenever $I,J\in \mathcal{D}_{\le n}$ satisfy $|I| = |J|$.
    \item\label{pro:diagonalization:item:3} If $T$ has $\delta$-large positive diagonal, then $D$ has $(1-\eta)\delta$-large positive diagonal.
  \end{enumerate}
\end{pro}

\begin{proof}
Fix $n\in \mathbb{N}_0$.
Let $m \in \mathbb{N}_0$ be the smallest integer for which
\begin{equation}\label{eq:8}
  2^m > \frac{2^{4(n+2)}\Gamma^4}{\eta_0^4},
  \qquad \text{where}\qquad
  \eta_0 = \frac{\eta\delta}{2^{4n+2}}.
\end{equation}
Put $M = 2^m$ and fix an integer $N\ge n + m$, i.e.,
\begin{equation*}
  N \ge n + 1 + \Bigl\lfloor 4(n+2) + 4 \log_2\Bigl( \frac{\Gamma}{\eta\delta} \Bigr) + 4(4n+2) \Bigr\rfloor = 21n + 17 + \Bigl\lfloor 4\log_2\Bigl(\frac{\Gamma}{\eta\delta}\Bigr) \Bigr\rfloor.
\end{equation*}
Moreover, let $T\colon Y_N\to Y_N$ be a linear operator with $\|T\|\le \Gamma$. We will construct a (finite) faithful Haar system $(\hat{h}_I)_{I\in \mathcal{D}_{\le n}}$ in~$Y_N$ which almost diagonalizes the operator~$T$. To this end, let $\theta = (\theta_K)_{K\in \mathcal{D}_{\le N}}$ be uniformly chosen from $\{ \pm 1 \}^{\mathcal{D}_{\le N}}$. We denote the uniform measure on $\{ \pm 1 \}^{\mathcal{D}_{\le N}}$ by~$\mathbb{P}$ and the corresponding expected value and variance by $\mathbb{E}$ and~$\mathbb{V}$, respectively. Moreover, for $K\in \mathcal{D}_{\le N}$, let $\mathbb{E}_K$ act on a function of $\theta$ by averaging it over $\theta_K = \pm 1$ (thus, $\mathbb{E}$ is the composition of all $\mathbb{E}_K$, $K\in \mathcal{D}_{\le N}$).

Now we inductively construct a randomized faithful Haar system $(\hat{h}_I(\theta))_{I\in \mathcal{D}_{\le n}}$ by putting $\mathcal{B}_{[0,1)}(\theta) = \mathcal{D}_{m}$ and
\begin{equation}\label{eq:1}
  \hat{h}_I(\theta) = \sum_{K\in \mathcal{B}_I(\theta)} \theta_Kh_K,
  \qquad I\in \mathcal{D}_{\le n},
\end{equation}
where
\begin{equation*}
  \mathcal{B}_{I^{\pm }}(\theta) = \bigl\{ K\in \mathcal{D}_{m + k + 1} : K\subset \{ \hat{h}_I(\theta) = \pm 1 \} \bigr\},
  \qquad I\in \mathcal{D}_k,\ k=0,1,\dots,n - 1.
\end{equation*}
Equivalently, this construction can be described as follows: First, put $\mathcal{B}_{[0,1)}(\theta) = \mathcal{D}_m$. After $\mathcal{B}_I(\theta)$ has been chosen for some $I\in \mathcal{D}_{<n}$, the successors $\{ K^{\pm } : K\in \mathcal{B}_I(\theta) \}$ are distributed evenly among $\mathcal{B}_{I^+}(\theta)$ and $\mathcal{B}_{I^-}(\theta)$: For every $K\in \mathcal{B}_I(\theta)$, if $\theta_K = 1$, then we add $K^+$ to $\mathcal{B}_{I^+}(\theta)$ and $K^-$ to $\mathcal{B}_{I^-}(\theta)$, and if $\theta = -1$, then we do the same with $K^+$ and $K^-$ swapped. The functions $(\hat{h}_I(\theta))_{I\in \mathcal{D}_{\le n}}$ are still given by equation~\eqref{eq:1}.

Now let $k,l\in \{ 0,\dots,n \}$ and fix two intervals $I\in \mathcal{D}_k$ and $J\in \mathcal{D}_l$. Let $\bar{K}_1,\dots,\bar{K}_M$ be an enumeration of the dyadic intervals in $\mathcal{D}_m$. Then we can write
\begin{equation*}
  \hat{h}_I(\theta) = \sum_{i=1}^M \theta_{K_i(\theta)}h_{K_i(\theta)}\qquad \text{and}\qquad
  \hat{h}_J(\theta) = \sum_{j=1}^M \theta_{L_j(\theta)}h_{L_j(\theta)},
\end{equation*}
where for each $i,j\in \{ 1,\dots,M \}$, $K_i(\theta)$ is a random subinterval of $\bar{K}_i$ and $L_j(\theta)$ is a random subinterval of $\bar{K}_j$. More precisely, for all $i$ and $j$, the interval~$K_i(\theta)$ is uniformly distributed on $\{ K\in \mathcal{D}_{m+k} : K\subset \bar{K}_i \}$, and $L_j(\theta)$ is uniformly distributed on $\{ L\in \mathcal{D}_{m+l} : L\subset \bar{K}_j \}$. For better readability, we abbreviate $K_i(\theta)$ and $L_j(\theta)$ by $K_i$ and $L_j$, respectively (still, these intervals depend on $\theta$). Note that for $j=i$, the intervals $K_i$ and $L_i$ are not independent (for example, they satisfy $K_i\subset L_i$ if and only if $I\subset J$ and, analogously, $K_i\supset L_i$ if and only if $I\supset J$).

Consider the random variable
 \begin{equation*}
   X_{I,J} = \langle \hat{h}_I(\theta), T \hat{h}_J(\theta) \rangle = \sum_{i,j=1}^M \theta_{K_i}\theta_{L_j} \langle h_{K_i}, T h_{L_j} \rangle.
 \end{equation*}
 We have
 \begin{equation*}
   X_{I,J}^2 = \sum_{i,i'=1}^M\sum_{j,j'=1}^M \theta_{K_i}\theta_{K_{i'}}\theta_{L_j}\theta_{L_{j'}} \langle h_{K_i}, Th_{L_j} \rangle \langle h_{K_{i'}}, Th_{L_{j'}} \rangle.
 \end{equation*}
 Next, we compute the expected value and upper bounds for the variance of $X_{I,J}$.

 \textbf{Expected value of $X_{I,J}$.}
 We assume that $k\ge l$ (the case $k<l$ is analogous). Since none of the intervals $K_1,\dots,K_M$ and $L_1,\dots,L_M$ depends on the signs $(\theta_K)_{K\in \mathcal{D}_{m+k}}$, we have
 \begin{equation}\label{eq:4}
   \mathbb{E} X_{I,J} = \mathbb{E} \sum_{i,j=1}^M \mathbb{E}_{K_i} (\theta_{K_i} \theta_{L_j}) \langle h_{K_i}, T h_{L_j} \rangle.
 \end{equation}
 If $I \ne J$, then equation~\eqref{eq:4} yields $\mathbb{E}X_{I,J} = 0$ since we always have $K_i\ne L_j$ for $i,j\in\{ 1,\dots,M \}$. On the other hand, if $I = J$, then we obtain
\begin{align*}
   \mathbb{E} X_{I,I} &= \mathbb{E} \sum_{i,j=1}^M \mathbb{E}_{K_i} (\theta_{K_i} \theta_{K_j}) \langle h_{K_i}, T h_{K_j} \rangle
    = \mathbb{E} \sum_{i=1}^M \langle h_{K_i}, T h_{K_i} \rangle\\
  &= \sum_{i=1}^M \ave_{\substack{K\in \mathcal{D}_{m+k}\\ K\subset \bar{K}_i}} \langle h_K, Th_K \rangle
  = 2^m \ave_{K\in \mathcal{D}_{m+k}} \langle h_K, Th_K \rangle
  = |I| \ave_{K\in \mathcal{D}_{m+k}} \frac{\langle h_K, Th_K \rangle}{|K|}.
\end{align*}

\textbf{Variance of $X_{I,J}$ for $I\ne J$.}
First, we assume that $I \ne J$ and $k\ge l$. Since $K_i\ne L_j$ for all $i,j$, and since none of the intervals $K_i$, $K_{i'}$, $L_j$, $L_{j'}$ depends on the signs $(\theta_K)_{K\in \mathcal{D}_{m+k}}$, we have
\begin{align*}
  \mathbb{E} X_{I,J}^2 &= \mathbb{E}\sum_{i,i'=1}^M\sum_{j,j'=1}^M \mathbb{E}_{K_i} \mathbb{E}_{K_{i'}} (\theta_{K_i}\theta_{K_{i'}})\, \theta_{L_j}\theta_{L_{j'}} \langle h_{K_i}, Th_{L_j} \rangle \langle h_{K_{i'}}, Th_{L_{j'}} \rangle\\
  &= \mathbb{E}\sum_{i=1}^M \sum_{j,j'=1}^M \theta_{L_j}\theta_{L_{j'}}\langle h_{K_i}, Th_{L_j} \rangle \langle h_{K_i}, Th_{L_{j'}} \rangle.
\end{align*}
Now consider fixed indices $i,j,j'$. If $j\ne j'$, then we must have $j\ne i$ or $j'\ne i$. Suppose that $j\ne i$. Then none of the intervals $L_j$, $L_{j'}$, $K_i$ depends on the signs $(\theta_L : L\in \mathcal{D}_{m+l},\, L\subset \bar{K}_j)$. Thus, we have
\begin{equation*}
  \mathbb{E} \bigl[ \theta_{L_j}\theta_{L_{j'}}\langle h_{K_i}, Th_{L_j} \rangle \langle h_{K_i}, Th_{L_{j'}} \rangle \bigr]
  = \mathbb{E} \bigl[  (\mathbb{E}_{L_j} \theta_{L_j})\theta_{L_{j'}}\langle h_{K_i}, Th_{L_j} \rangle \langle h_{K_i}, Th_{L_{j'}} \rangle \bigr]
    = 0.
\end{equation*}
The same holds if $j' \ne i$. Hence, we obtain
\begin{equation}\label{eq:3}
  \mathbb{E}X_{I,J}^2 = \mathbb{E} \sum_{i,j=1}^M \langle h_{K_i}, Th_{L_j} \rangle^2.
\end{equation}
If $k\le l$, then equation~\eqref{eq:3} holds by analogous arguments.

Following~\cite{MR3990955}, we will now estimate~\eqref{eq:3} in two different ways. We fix $\theta$, put $a_{K_i,L_j} = \langle h_{K_i}, Th_{L_j} \rangle$ for $i,j=1,\dots,M$ and note that
\begin{equation*}
  |a_{K_i,L_j}|\le \|h_{K_i}\|_{Y^{*}}\|T\| \|h_{L_j}\|_Y.
\end{equation*}
By the $1$-unconditionality of $(h_{L_j})_{j=1}^M$ in $Y$, we have
\begin{align*}
  \sum_{i=1}^M\sum_{j=1}^M \langle h_{K_i}, Th_{L_j} \rangle^2
  &= \sum_{i=1}^M \Big\langle h_{K_i}, T\Big( \sum_{j=1}^M a_{K_i,L_j} h_{L_j} \Big) \Big\rangle\\
  &\le \sum_{i=1}^M \|h_{K_i}\|_{Y^{*}} \|T\| \max_{j=1,\dots,M} |a_{K_i,L_j}| \Big\| \sum_{j=1}^M h_{L_j} \Big\|_Y\\
  &\le \|T\|^2\sum_{i=1}^M \|h_{K_i}\|_{Y^{*}}^2\max_{j=1,\dots,M}\|h_{L_j}\|_Y.
\end{align*}
Now we use the inequalities
\begin{equation*}
  \|h_{L_j}\|_Y\le \|h_{\bar{K}_1}\|_Y\qquad
  \text{and}\qquad
  \|h_{K_i}\|_{Y^{*}}^2 = \|h_{K_i}\|_{Y^{*}}\frac{|K_i|}{\|h_{K_i}\|_Y}\le\|h_{\bar{K}_1}\|_{Y^{*}} \frac{|K_i|}{|I|\|h_{\bar{K}_1}\|_Y}
\end{equation*}
to conclude that
\begin{equation}\label{eq:5}
  \sum_{i=1}^M\sum_{j=1}^M \langle h_{K_i}, Th_{L_j} \rangle^2
  \le \|T\|^2 \|h_{\bar{K}_1}\|_{Y^{*}} \sum_{i=1}^M \frac{|K_i|}{|I|} = \|T\|^2 \|h_{\bar{K}_1}\|_{Y^{*}}.
\end{equation}
On the other hand, we can rewrite~\eqref{eq:3} using
\begin{equation*}
  \sum_{i=1}^M\sum_{j=1}^M \langle h_{K_i}, Th_{L_j} \rangle^2 = \sum_{j=1}^M \Big\langle \sum_{i=1}^M a_{K_i,L_j} h_{K_i}, Th_{L_j} \Big\rangle.
\end{equation*}
Then we perform an analogous computation to the one above, using \Cref{pro:HSHS-1d}~\eqref{pro:HSHS-1d:v} (note that $\|\cdot \|_{Y^{*}}\le \|\cdot \|_{X(\mathbf{r})^{*}}$) as well as \Cref{pro:HSHS-1d}~\eqref{pro:HSHS-1d:iii} and \Cref{pro:HS-1d}~\eqref{pro:HS-1d:iii}. In the end, we obtain
\begin{equation}\label{eq:6}
  \sum_{i=1}^M\sum_{j=1}^M \langle h_{K_i}, Th_{L_j} \rangle^2 \le \|T\|^2\|h_{\bar{K}_1}\|_Y.
\end{equation}
Combining~\eqref{eq:5} and~\eqref{eq:6} yields
\begin{equation}\label{eq:var-IJ}
  \mathbb{V}X_{I,J}
  = \mathbb{E} X_{I,J}^2
  \le \|T\|^2 \sqrt{\|h_{\bar{K}_1}\|_Y\|h_{\bar{K}_1}\|_{Y^{*}}}
  = \|T\|^2 \sqrt{|\bar{K}_1|}
  \le \Gamma^2 2^{-m/2}.
\end{equation}

\textbf{Variance of $X_{I,I}$.}
We have
\begin{align*}
  \mathbb{E}X_{I,I}^2
  &= \mathbb{E}\sum_{i,i'=1}^M\sum_{j,j'=1}^M \theta_{K_i}\theta_{K_{i'}}\theta_{K_j}\theta_{K_{j'}} \langle h_{K_i}, Th_{K_j} \rangle\langle h_{K_{i'}}, Th_{K_{j'}} \rangle\\
  &= \mathbb{E}\sum_{i,i'=1}^M\sum_{j,j'=1}^M\mathbb{E}_{K_i}\mathbb{E}_{K_{i'}}\mathbb{E}_{K_j}\mathbb{E}_{K_{j'}} (\theta_{K_i}\theta_{K_{i'}}\theta_{K_j}\theta_{K_{j'}})\, \langle h_{K_i}, Th_{K_j} \rangle\langle h_{K_{i'}}, Th_{K_{j'}} \rangle.
\end{align*}
Now observe that we have
\begin{equation*}
 \mathbb{E}_{K_i}\mathbb{E}_{K_{i'}}\mathbb{E}_{K_j}\mathbb{E}_{K_{j'}} (\theta_{K_i}\theta_{K_{i'}}\theta_{K_j}\theta_{K_{j'}}) = 1
\end{equation*}
if one of the following conditions holds, and the expression is $0$ otherwise (cf.~\cite{MR3990955}):
\begin{enumerate}[(i)]\itemsep0em
  \item $i = j = i' = j'$
  \item $i = j \ne i' = j'$
  \item $i = i' \ne j = j'$
  \item $i = j' \ne i' = j$.
\end{enumerate}
Hence, we can write $\mathbb{E}X_{I,I}^2 = V_1 + V_2 + V_3 + V_4$, where
\begin{align*}
  V_1 &= \mathbb{E}\sum_{i=1}^M \langle h_{K_i}, Th_{K_i} \rangle^2,\\
  V_2 &= \mathbb{E}\sum_{i=1}^M\sum_{\substack{j=1\\j\ne i}}^M \langle h_{K_i}, Th_{K_i} \rangle\langle h_{K_j}, Th_{K_j} \rangle
  = \sum_{i=1}^M\sum_{\substack{j=1\\j\ne i}}^M \bigl(\mathbb{E}\langle h_{K_i}, Th_{K_i} \rangle\bigr)\bigl(\mathbb{E}\langle h_{K_j}, Th_{K_j} \rangle\bigr),\\
  V_3 &= \mathbb{E}\sum_{i=1}^M\sum_{\substack{j=1\\j\ne i}}^M \langle h_{K_i}, Th_{K_j} \rangle^2,\\
  V_4 &= \mathbb{E}\sum_{i=1}^M\sum_{\substack{j=1\\j\ne i}}^M \langle h_{K_i}, Th_{K_j} \rangle\langle h_{K_j}, Th_{K_i} \rangle.
\end{align*}
In general, we expect $V_2$ to be large, whereas the other terms will be small. Note that
\begin{equation*}
  (\mathbb{E} X_{I,I})^2
  = \Bigl(\sum_{i=1}^M \mathbb{E}\langle h_{K_i}, T h_{K_i} \rangle \Bigr)^2
  \ge V_2,
\end{equation*}
and thus,
\begin{equation}\label{eq:7}
  \mathbb{V}X_{I,I} \le V_1 + V_3 + V_4.
\end{equation}
Hence, it suffices to give upper bounds for $V_1$, $V_3$ and $V_4$. We have
\begin{equation*}
  V_1 \le \mathbb{E}\sum_{i=1}^M \|h_{K_i}\|_{Y^{*}}^2\|T\|^2\|h_{K_i}\|_Y^2
        = \|T\|^2\, \mathbb{E}\sum_{i=1}^M|K_i|^2 = \|T\|^2|I|^22^{-m}.
\end{equation*}
Moreover, the double sums in $V_3$ and $V_4$ can be estimated using the same arguments as in the computation of $\mathbb{V} X_{I,J}$ for $I\ne J$, yielding
\begin{equation*}
  V_3\le \|T\|^2 2^{-m/2}\qquad \text{and}\qquad V_4\le \|T\|^2 2^{-m/2}.
\end{equation*}
By plugging these estimates into~\eqref{eq:7}, we obtain
\begin{equation}\label{eq:var-II}
  \mathbb{V}X_{I,I}\le 3\|T\|^2 2^{-m/2}\le 3\Gamma^2 2^{-m/2}.
\end{equation}

\textbf{Conclusion of the proof.}
As in~\cite{MR3990955}, we consider the off-diagonal events
\begin{equation*}
  O_{I,J} = \{ \theta\in \{ \pm 1 \}^{\mathcal{D}_{\le N}} : |X_{I,J}(\theta)|\ge \eta_0 \},\qquad I, J\in \mathcal{D}_{\le n},\ I\ne J
\end{equation*}
and the diagonal events
\begin{equation*}
  D_I = \{ \theta \in \{ \pm 1 \}^{\mathcal{D}_{\le N}} : |X_{I,I}(\theta) - \mathbb{E}X_{I,I}| \ge \eta_0 \},\qquad I\in \mathcal{D}_{\le n}.
\end{equation*}
Using Chebyshev's inequality and our upper bounds~\eqref{eq:var-IJ} and~\eqref{eq:var-II} for $\mathbb{V}X_{I,J}$ and $\mathbb{V}X_{I,I}$, we obtain
\begin{equation*}
  \mathbb{P}(O_{I,J})\le \frac{\Gamma^2}{2^{m/2}\eta_0^2}\quad \text{and} \quad \mathbb{P}(D_I)\le \frac{3\Gamma^2}{2^{m/2}\eta_0^2},
  \qquad I, J\in \mathcal{D}_{\le n},\ I\ne J.
\end{equation*}
Hence, by~\eqref{eq:8}, we have
\begin{equation*}
  \mathbb{P}\Big( \bigcup_{\substack{I,J\in \mathcal{D}_{\le n}\\ I\ne J}}O_{I,J}\cup \bigcup_{I\in \mathcal{D}_{\le n}} D_I \Big)
  \le \sum_{\substack{I,J\in \mathcal{D}_{\le n}\\ I\ne J}} \mathbb{P}(O_{I,J}) + \sum_{I\in \mathcal{D}_{\le n}} \mathbb{P}(D_I)
  \le \frac{2^{2(n+2)}\Gamma^2}{2^{m/2}\eta_0^2}< 1.
\end{equation*}
Thus, we can find at least one family of signs $\theta \in \{ \pm 1 \}^{\mathcal{D}_{\le N}}$ such that the corresponding faithful Haar system $(\hat{h}_I)_{I\in \mathcal{D}_{\le n}} = (\hat{h}_I(\theta))_{I\in \mathcal{D}_{\le n}}$ satisfies
\begin{align}
  \frac{|\langle \hat{h}_I, T \hat{h}_J \rangle|}{|I|}
  &\le \frac{\eta_0}{|I|}\le 2^n\eta_0,
  \qquad I,J\in \mathcal{D}_{\le n},\ I\ne J,\label{eq:10}\\
  \Bigl|\frac{\langle \hat{h}_I, T \hat{h}_I \rangle}{|I|} - \ave_{|K| = 2^{-m}|I|} \frac{\langle h_K, Th_K \rangle}{|K|}\Bigr|
  &\le \frac{\eta_0}{|I|}
  \le 2^n\eta_0,
  \qquad I\in \mathcal{D}_{\le n}.\label{eq:11}
\end{align}

Now let $\hat{A}\colon Y_N\to Y_n$ and $\hat{B}\colon Y_n\to Y_N$ denote the operators associated with $(h_I)_{I\in \mathcal{D}_{\le n}}$ as defined in \Cref{pro:operators-Ahat-Bhat}. Put
\begin{equation*}
  d_I = \frac{\langle \hat{h}_I, T \hat{h}_I \rangle}{|I|}, \qquad I\in \mathcal{D}_{\le n},
\end{equation*}
and consider the Haar multiplier $D\colon Y_n\to Y_n$ defined as the linear extension of $Dh_I = d_Ih_I$, $I\in \mathcal{D}_{\le n}$. We have $|d_I| \le \|T\|$ for all $I\in \mathcal{D}_{\le n}$ (this follows, e.g., from \Cref{pro:operators-Ahat-Bhat} and \Cref{pro:HSHS-1d}~\eqref{pro:HSHS-1d:single-layer}).
Observe that by the inequalities~\eqref{eq:11} and $2^n\eta_0\le 8^{-n}\eta\delta$ (cf.~\eqref{eq:8}), we have $|d_I - d_J|\le 8^{-n}\eta\delta$ whenever $|I| = |J|$. Moreover, if $T$ has $\delta$-large positive diagonal, i.e., $\langle h_K, Th_K \rangle\ge \delta|K|$ for all $K\in \mathcal{D}_{\le N}$, then inequality~\eqref{eq:11} together with $2^n\eta_0\le \eta\delta$ implies that $d_I\ge (1 - \eta)\delta$ for all $I\in \mathcal{D}_{\le n}$.

Finally, let $x = \sum_{J\in \mathcal{D}_{\le n}} a_Jh_J\in Y_n$. Then, using~\eqref{eq:10} and the inequality $|a_J|\le \|x\|_Y/\|h_J\|_Y\le 2^n\|x\|_Y$ for $J\in \mathcal{D}_{\le n}$ (see \Cref{pro:HSHS-1d}~\eqref{pro:HSHS-1d:single-layer}), we obtain
\begin{align*}
  \|(\hat{A}T \hat{B} - D)x\|_Y &= \Big\| \sum_{J\in \mathcal{D}_{\le n}} \sum_{\substack{I\in \mathcal{D}_{\le n}\\ I\ne J}} a_J \frac{\langle \hat{h}_I, T \hat{h}_J \rangle}{|I|} h_I \Big\|_Y\\
  &\le 2^n\eta_0\sum_{J\in \mathcal{D}_{\le n}} \sum_{\substack{I\in \mathcal{D}_{\le n}\\ I\ne J}} |a_J|\|h_I\|_Y
  \le 2^{4n+2}\eta_0\|x\|_Y,
\end{align*}
and hence, by~\eqref{eq:8},
\begin{equation*}
  \|\hat{A}T \hat{B} - D\| \le 2^{4n+2}\eta_0 \le \eta\delta.\qedhere
\end{equation*}
\end{proof}

\section{Stabilization of Haar multipliers}\label{sec:stabilization}

Our next goal is to reduce the Haar multiplier~$D$ obtained from~\Cref{pro:diagonalization} to a constant multiple of the identity. This is achieved by first \emph{stabilizing}~$D$ and then employing a perturbation argument. The infinite-dimensional stabilization methods from~\cite{MR4430957,LechnerSpeckhofer2023,LechnerMotakisMuellerSchlumprecht2023} involving cluster points and probabilistic methods will be replaced by a finite, combinatorial version which just uses the pigeonhole principle. Note that in the stabilization part, we will also make use of the fact that the entries of $D$ are already stable along every level of the dyadic tree (see \Cref{pro:diagonalization}~\eqref{pro:diagonalization:item:2}).

\begin{pro}\label{pro:stabilization}
  Let $Y$ be a Haar system Hardy space, and let $\Gamma,\delta,\eta > 0$. Moreover, let $n, N\in \mathbb{N}_0$ be chosen such that
  \begin{equation}\label{eq:110}
    N \ge 42n(n+1)\biggl\lceil \frac{\Gamma}{\eta\delta} \biggr\rceil + 42 + \Bigl\lfloor 4\log_2\Bigl(\frac{\Gamma}{\eta\delta}\Bigr) \Bigr\rfloor.
  \end{equation}
  Then for every linear operator $T\colon Y_N\to Y_N$ with $\|T\|\le \Gamma$, there exists a scalar~$c$ with $|c|\le \|T\|$ such that $cI_{Y_n}$ projectionally factors through~$T$ with constant~$1$ and error~$3\eta\delta$. Moreover, if $T$ has $\delta$-large positive diagonal, then we have $c\ge (1-\eta)\delta$.
  If the Haar system is $K$-unconditional in~$Y$ for some $K\ge 1$, then inequality~\eqref{eq:110} can be replaced by
  \begin{equation*}
    N \ge 42n\biggl\lceil \frac{K\Gamma}{\eta\delta} \biggr\rceil + 42 + \Bigl\lfloor 4\log_2\Bigl(\frac{\Gamma}{\eta\delta}\Bigr) \Bigr\rfloor.
  \end{equation*}
\end{pro}

\begin{proof}
Let $n\in \mathbb{N}_0$ and fix $\Gamma,\delta,\eta > 0$. Put
\begin{equation}\label{eq:111}
 \tilde{n} = 2n(n+1) \biggl\lceil \frac{\Gamma}{\eta\delta} \biggr\rceil + 1 \ge n\biggl\lceil \frac{2\Gamma(n+1)}{\eta\delta} \biggr\rceil + 1
\end{equation}
and let $N\in \mathbb{N}_0$ satisfy~\eqref{eq:110}, i.e., $N\ge 21(\tilde{n}+1) + \lfloor 4\log_2(\frac{\Gamma}{\eta\delta}) \rfloor$. Let $T\colon Y_N\to Y_N$ be a linear operator with $\|T\|\le \Gamma$. By \Cref{pro:diagonalization}, there exists a Haar multiplier $D\colon Y_{\tilde{n}}\to Y_{\tilde{n}}$ with entries $(d_K)_{K\in \mathcal{D}_{\le \tilde{n}}}$ such that $D$ projectionally factors through~$T$ with constant~$1$ and error~$\eta\delta$, and such that $|d_K|\le \|T\|\le \Gamma$ for all $K\in \mathcal{D}_{\le \tilde{n}}$ and
\begin{equation}\label{eq:9}
  |d_K - d_L|\le 8^{-\tilde{n}}\eta\delta\le 8^{-n}\eta\delta\quad \text{whenever }\quad |K| = |L|.
\end{equation}
Moreover, if $T$ has $\delta$-large positive diagonal, then we have $d_I \ge (1-\eta)\delta$ for all $I\in \mathcal{D}_{\le \tilde{n}}$.
Now consider the entries $d_{[0,2^{-k})}$, $0\le k\le \tilde{n}$. Divide the interval $[-\Gamma,\Gamma]$ into $\lceil 2\Gamma(n+1)/(\eta\delta) \rceil$ subintervals of length at most $\eta\delta/(n+1)$. Then, by inequality~\eqref{eq:111} and the pigeonhole principle, we can find natural numbers $0\le k_0 < k_1 < \dots < k_n \le \tilde{n}$ such that
\begin{equation}\label{eq:113}
  |d_{[0,2^{-k_i})} - d_{[0,2^{-k_0})}| \le \frac{\eta\delta}{n+1},\qquad i = 0,\dots, n.
\end{equation}
Hence, after putting $c = d_{[0,2^{-k_0})}$, we have $|c|\le \|T\|$ and, by~\eqref{eq:9} and~\eqref{eq:113},
\begin{equation}\label{eq:12}
  |d_K - c|\le \frac{2\eta\delta}{n+1},\qquad K\in \mathcal{D}_{k_0}\cup \dots\cup \mathcal{D}_{k_n}.
\end{equation}
If $T$ has $\delta$-large positive diagonal, then we have $c \ge (1-\eta)\delta$.
Now let $(\hat{h}_I)_{I\in \mathcal{D}_{\le n}}$ be a faithful Haar system with frequencies $k_0,\dots,k_n$, i.e., for every $I\in \mathcal{D}_i$ ($0\le i\le n$), we have $\mathcal{B}_I\subset \mathcal{D}_{k_i}$, where $\mathcal{B}_I$ is the Haar support of $\hat{h}_I$. Let $\hat{A}\colon Y_{\tilde{n}}\to Y_n$ and $\hat{B}\colon Y_n\to Y_{\tilde{n}}$ denote the associated operators as defined in \Cref{pro:operators-Ahat-Bhat}. Then $D^{\mathrm{stab}} = \hat{A} D \hat{B}$ is also a Haar multiplier, and its entries $(d^{\mathrm{stab}}_I)_{I\in \mathcal{D}_{\le n}}$ satisfy
\begin{equation*}
  d^{\mathrm{stab}}_I = \frac{\langle \hat{h}_I, D \hat{h}_I \rangle}{|I|} = \sum_{K\in \mathcal{B}_I} d_K\frac{|K|}{|I|},\qquad I\in \mathcal{D}_{\le n}.
\end{equation*}
Together with inequality~\eqref{eq:12}, this implies that
\begin{equation}\label{eq:120}
  |d^{\mathrm{stab}}_I - c| \le \frac{2\eta\delta}{n + 1},\qquad I\in \mathcal{D}_{\le n}.
\end{equation}
We will now show that $\|D^{\mathrm{stab}} - cI_{Y_n}\|\le 2\eta\delta$. To this end, let $x = \sum_{I\in \mathcal{D}_{\le n}} a_Ih_I\in Y_n$. Then, exploiting the $1$-unconditionality of $(h_I)_{I\in \mathcal{D}_k}$ in~$Y$ for each $k\in \mathbb{N}_0$ and using \Cref{pro:HSHS-1d}~\eqref{pro:HSHS-1d:single-layer}, we obtain
\begin{align*}
  \|(D^{\mathrm{stab}} - cI_{Y_n})x\|_Y
  &= \Big\| \sum_{I\in \mathcal{D}_{\le n}} (d^{\mathrm{stab}}_I - c)a_Ih_I \Big\|_Y
    \le \sum_{k=0}^n \Big\| \sum_{I\in \mathcal{D}_k}(d_I^{\mathrm{stab}} - c)a_Ih_I \Big\|_Y\\
  &\le \frac{2\eta\delta}{n+1}\sum_{k=0}^n \Big\| \sum_{I\in \mathcal{D}_k} a_Ih_I \Big\|_Y
    \le \frac{2\eta\delta}{n+1}\sum_{k=0}^n \|x\|_Y = 2\eta\delta\|x\|_Y,
\end{align*}
i.e., $\|D^{\mathrm{stab}} - cI_{Y_n}\|\le 2\eta\delta$. Recall that $D$ projectionally factors through~$T$ with constant~$1$ and error~$\eta\delta$ and $D^{\mathrm{stab}}$ projectionally factors through~$D$ with constant~$1$ and error~$0$. By combining all these statements, we obtain that $cI_{Y_n}$ projectionally factors through~$T$ with constant~$1$ and error~$3\eta\delta$.

If the Haar system is $K$-unconditional in~$Y$, then we replace~\eqref{eq:111} by
\begin{equation*}
 \tilde{n} = 2n \biggl\lceil \frac{K\Gamma}{\eta\delta} \biggr\rceil + 1 \ge n\biggl\lceil \frac{2K\Gamma}{\eta\delta} \biggr\rceil + 1,
\end{equation*}
and we also replace the factor~$n+1$ by~$K$ in the pigeonhole argument and in the denominator of~\eqref{eq:120}. This inequality then directly implies that $\|D^{\mathrm{stab}} - cI_{Y_n}\|\le 2\eta\delta$, and the result follows as above.
\end{proof}

\begin{rem}
  If the space under consideration is~$L^1$, then of course, the results by Bourgain and Tzafriri~\cite{MR0890420} imply factorization results with linear dimension dependence (see~\Cref{sec:introduction-and-main-results}). However, our method can also be modified to yield a better result in~$L^1$ in the sense that inequality~\eqref{eq:110} is replaced by a bound of the form $N\ge C(\Gamma,\delta,\eta)\, n$: Replace the definition of~$\tilde{n}$ in the proof of \Cref{pro:stabilization} by $\tilde{n} = l\cdot n$ for some $l\ge 1$ to be determined later, and let $D\colon Y_{\tilde{n}}\to Y_{\tilde{n}}$ be the Haar multiplier obtained from \Cref{pro:diagonalization}. Write $d_k = d_{[0,2^{-k})}$ for $k=0,\dots,\tilde{n}$. Moreover, let $\bar{D}\colon Y_{\tilde{n}}\to Y_{\tilde{n}}$ be the Haar multiplier defined by $\bar{D}h_I = d_kh_I$ for all $I\in \mathcal{D}_k$, $0\le k\le \tilde{n}$. By~\eqref{eq:9} and the triangle inequality, we have $\|D - \bar{D}\|\le 2\eta\delta$. Next, we use a result by Semenov and Uksusov~\cite{MR2975943} (see also \cite{MR3397275,MR4430957} for alternative proofs), which characterizes the bounded Haar multipliers $M\colon L^1\to L^1$ by identifying a quantity $\vvvert M \vvvert$ that is equivalent to the operator norm of~$M$: In the formulation of~\cite[Theorem~2.6]{MR4430957}, we have
  \begin{equation*}
    \|M\| \sim \vvvert M \vvvert := \sup_{(I_k)} \Big( \sum_{k=0}^{\infty}|m_{I_k} - m_{I_{k+1}}| + \limsup_{k\to \infty}|m_{I_k}| \Big),
  \end{equation*}
  where the supremum is taken over all sequences $(I_k)_{k=0}^{\infty}$ in $\mathcal{D}^+$ such that $I_0 = \varnothing$, $I_1 = [0,1)$, and $I_{k+1} \in \{ I_k^+,I_k^- \}$ for every $k\ge 1$. A straightforward argument shows that we can restrict the domain to $Y_{\tilde{n}}$ and apply this result to $\bar{D}$ to obtain
  \begin{equation}\label{eq:72}
    \|\bar{D}\| \sim |d_0 - d_1| + \dots + |d_{\tilde{n}-1} - d_{\tilde{n}}| + |d_{\tilde{n}}|.
  \end{equation}
  Since $\|\bar{D}\|\le \Gamma+3\eta\delta$, we can find $0\le s\le \tilde{n} - n$ such that
  \begin{equation*}
    \frac{\Gamma + 3\eta\delta}{l} \gtrsim |d_s - d_{s+1}| + \dots + |d_{s+n-1} - d_{s+n}|.
  \end{equation*}
  By constructing a faithful Haar system $(\hat{h}_I)_{I\in \mathcal{D}_{\le n}}$ with frequencies $s,s+1,\dots,s+n$ and using the associated operators $\hat{A},\hat{B}$, we obtain a Haar multiplier $D^{\mathrm{stab}} = \hat{A}\bar{D}\hat{B}$, which satisfies $D^{\mathrm{stab}}h_I = d_{s+k}h_I$ for all $I\in \mathcal{D}_k$, $0\le k\le n$. Thus, applying~\eqref{eq:72} to~$D^{\mathrm{stab}}$ and choosing, for example, $l \ge \frac{\Gamma}{3\eta\delta} + 1$ yields $\|D^{\mathrm{stab}} - d_{s+n}I_{Y_n}\|\lesssim \eta\delta$.
\end{rem}
Despite the fact that the dimension dependence can be improved in the unconditional case and also in the case of~$L^1$ (and, by duality, in~$L^{\infty}$), we do not know the answer to the following question:
\begin{question}
Is it true that for every Haar system Hardy space~$Y$, an inequality of the form $N \ge C(\Gamma,\delta,\eta,Y)\, n$ is sufficient for the conclusion of \Cref{pro:stabilization} (and hence \Cref{thm:main-result}) to hold?
\end{question}

Using the stabilization result \Cref{pro:stabilization}, we are now able to prove our main result.

\begin{proof}[Proof of \Cref{thm:main-result}]
  We first prove~\eqref{thm:main-result:ii}. Suppose that the linear operator $T\colon Y_N\to Y_N$ with $\|T\|\le \Gamma$ has $\delta$-large positive diagonal with respect to the Haar system. Then by \Cref{pro:stabilization}, we can find a scalar $c\ge (1-\eta)\delta$ and operators $A\colon Y_N\to Y_n$ and $B\colon Y_n\to Y_N$ such that $\|A\|\|B\|\le 1$ and $\|cI_{Y_n} - ATB\|\le 3\eta\delta$. Thus,
  \begin{equation*}
    \|I_{Y_n} - \tfrac{1}{c}ATB\| \le \frac{3\eta\delta}{(1-\eta)\delta} =  \frac{3\eta}{1-\eta} < 1.
  \end{equation*}
  This implies that the operator $Q = \frac{1}{c}ATB$ is invertible, and its inverse satisfies
  \begin{equation*}
    \|Q^{-1}\|\le \frac{1}{1 - 3\eta/(1-\eta)}.
  \end{equation*}
  Hence, we have the factorization $I_{Y_n} = \frac{1}{c}ATBQ^{-1}$, where
  \begin{equation*}
    \|\tfrac{1}{c}A\|\|BQ^{-1}\|\le \frac{\|Q^{-1}\|}{c}\le \frac{1}{\delta(1 - 4\eta)} \le \frac{1 + \varepsilon}{\delta}.
  \end{equation*}

  To prove~\eqref{thm:main-result:i}, we assume that $\delta = 1$. By \Cref{pro:stabilization}, we can find a scalar~$c$ with $|c|\le \|T\|$ such that $cI_{Y_n}$ \emph{projectionally} factors through $T$ with constant~$1$ and error~$3\eta$. Thus, we also know that $(1-c)I_{Y_n}$ factors through $I_{Y_N} - T$ with the same constant and error. Note that we either have $c\ge \frac{1}{2}$ or $1 - c \ge \frac{1}{2}$. In the latter case, we replace $T$ by $I_{Y_N} - T$ and $c$ by $1 - c$. Then the proof can be completed in the same manner as above, with slightly modified estimates and an additional factor~$2$ in the factorization constant.

  The results for spaces in which the Haar system is unconditional follow from the last statement in \Cref{pro:stabilization}.
\end{proof}

\section{Reduction to positive diagonal}\label{sec:reduction-to-positive-diagonal}

Finally, we consider the case of an operator $\widetilde{T}\colon Y_{\widetilde{N}}\to Y_{\widetilde{N}}$ with $\delta$-large, not necessarily positive diagonal. If the Haar system is $K$-unconditional in~$Y$, then by composing~$\widetilde{T}$ with a Haar multiplier with entries $\pm 1$, we obtain an operator $T\colon Y_{\widetilde{N}}\to Y_{\widetilde{N}}$ with $\delta$-large \emph{positive} diagonal such that $T$ factors through~$\widetilde{T}$ with constant~$K$. Then \Cref{thm:main-result} can be applied to $T$.

In the general case, however, a more sophisticated construction is needed to reduce~$\widetilde{T}$ to an operator $T\colon Y_N\to Y_N$ with positive diagonal. We employ a discrete version of the \emph{Gamlen-Gaudet construction}. This technique goes back to~\cite{MR0328575} (see also~\cite{MR2157745}), and the infinite version was also utilized, for example, in~\cite{MR4145794,LechnerSpeckhofer2023}.
Due to the combinatorial arguments used in this method, we need to assume an inequality of the form $\widetilde{N}\ge C(\varepsilon) N^22^N$.

\begin{pro}\label{pro:positive-reduction}
  Let $Y$ be a Haar system Hardy space, and let $\varepsilon,\delta > 0$. Moreover, let $N,\widetilde{N}\in \mathbb{N}_0$ be chosen such that
  \begin{equation}\label{eq:13}
    \widetilde{N}\ge 2N \bigg\lceil \frac{N}{\varepsilon} + 1 \bigg\rceil 2^N.
  \end{equation}
  Then for every linear operator $\widetilde{T}\colon Y_{\widetilde{N}}\to Y_{\widetilde{N}}$ with $\delta$-large diagonal, there exists a linear operator $T\colon Y_N\to Y_N$ with $\delta$-large \emph{positive} diagonal such that $T$ factors through~$\widetilde{T}$ with constant~$2(1+\varepsilon)$.
\end{pro}

\begin{proof}
  Suppose that $\widetilde{N} \ge 2lN$ for some $l\in \mathbb{N}$ to be determined later, and let $\widetilde{T}\colon Y_{\widetilde{N}}\to Y_{\widetilde{N}}$ be a linear operator with $\delta$-large diagonal (with respect to the Haar basis). Put
  \begin{equation*}
    \mathcal{A}_1 = \{ K\in \mathcal{D}_{\le \widetilde{N}} : \langle h_K, \widetilde{T}h_K \rangle \ge \delta|K| \},\qquad
    \mathcal{A}_2 = \{ K\in \mathcal{D}_{\le \widetilde{N}} : \langle h_K, \widetilde{T}h_K \rangle \le -\delta|K| \}.
  \end{equation*}
  Thus, $\mathcal{A}_1$ and $\mathcal{A}_2$ form a partition of $\mathcal{D}_{\le \widetilde{N}}$. Observe that for every $K\in \mathcal{D}_{\widetilde{N}}$, there exists some $i_K\in \{ 1,2 \}$ such that there are more than $\widetilde{N}/2\ge lN$ intervals in~$\mathcal{A}_{i_K}$ which contain~$K$. Moreover, there exists $i\in \{ 1,2 \}$ such that in total, the intervals $K\in \mathcal{D}_{\widetilde{N}}$ with $i_K = i$ cover a set of measure at least $1/2$. We may assume without loss of generality that $i = 1$ (otherwise, we replace $\widetilde{T}$ by~$-\widetilde{T}$). Now, for a given collection of dyadic intervals $\mathcal{A}\subset \mathcal{D}$, we inductively define the generations
  \begin{align*}
    \mathcal{G}_0(\mathcal{A}) &= \{ K\in \mathcal{A} : K \text{ is maximal with respect to inclusion} \},\\
    \mathcal{G}_k(\mathcal{A}) &= \mathcal{G}_0\Big( \mathcal{A}\setminus \bigcup_{j=0}^{k-1}\mathcal{G}_j(\mathcal{A}) \Big),\qquad k\ge 1.
  \end{align*}
  We will use the abbreviation $\mathcal{G}_k = \mathcal{G}_k(\mathcal{A}_1)$, $k\ge 0$.
  Thus, $\mathcal{G}_0,\mathcal{G}_1,\dots$ are pairwise disjoint collections of pairwise disjoint dyadic intervals, and for every~$k$, the collection~$\mathcal{G}_k$ consists of those intervals $K\in \mathcal{A}_1$ which are strictly contained in exactly $k$ intervals of $\mathcal{A}_1$.
  In particular, we have $\mathcal{G}_0\supset \mathcal{G}_1\supset \cdots$, and if $k\ge 1$, then for every $K\in \mathcal{G}_k$, there exists an interval $L\in \mathcal{G}_{k-1}$ such that $K\subset L^+$ or $K\subset L^-$.
  See \Cref{fig:collections-Gk} for an example.
\begin{figure}[t]
  \centering \includegraphics{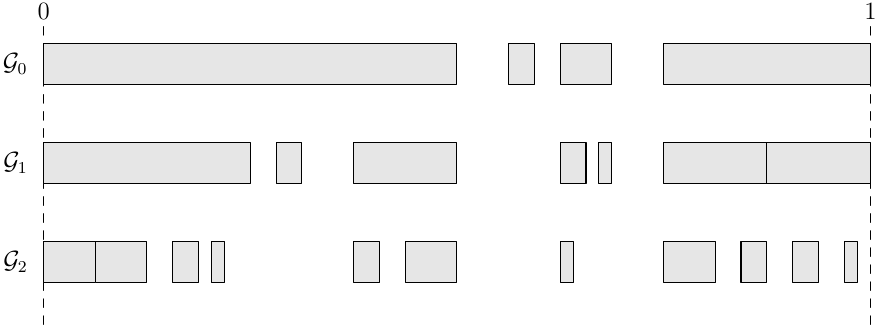}
  \caption{The collections $\mathcal{G}_k$}
  \label{fig:collections-Gk}
\end{figure}

 By the above considerations, we have $|\mathcal{G}_{lN}^{*}|\ge 1/2 \ge 1/2\, |\mathcal{G}_0^{*}|$. Thus,
  \begin{equation*}
    \frac{|\mathcal{G}_{lN}^{*}|}{|\mathcal{G}_{(l-1)N}^{*}|}
    \cdots
    \frac{|\mathcal{G}_{2N}^{*}|}{|\mathcal{G}_N^{*}|}
    \cdot \frac{|\mathcal{G}_N^{*}|}{|\mathcal{G}_0^{*}|} \ge \frac{1}{2},
  \end{equation*}
  and so there exists some $0\le s\le (l-1)N$ such that $|\mathcal{G}_{s+N}^{*}| \ge 2^{-1/l}|\mathcal{G}_s^{*}|$. By choosing
  \begin{equation*}
    l \ge \Bigl( \frac{N}{\varepsilon} + 1 \Bigr)2^N\ln 2 \ge -\frac{1}{\log_2\bigl( 1 - \frac{1}{(N/\varepsilon+1) 2^N} \bigr)},
  \end{equation*}
  we obtain
  \begin{equation}\label{eq:14}
    |\mathcal{G}_{s+N}^{*}| \ge \Bigl( 1 - \frac{1}{(N/\varepsilon+1)2^N} \Bigr) |\mathcal{G}_s^{*}|.
  \end{equation}

  Next, we construct an almost faithful Haar system $(\tilde{h}_I)_{I\in \mathcal{D}_{\le N}}$ by alternately choosing collections $\mathcal{B}_I\subset \mathcal{A}_1$ and signs $(\theta_K)_{K\in \mathcal{B}_I}$: First, we put $\mathcal{B}_{[0,1)} = \mathcal{G}_s$. If $\mathcal{B}_I\subset \mathcal{A}_1$ has already been constructed for some $I\in \mathcal{D}_k$, $0\le k\le N$, then we choose the signs $(\theta_K)_{K\in \mathcal{B}_I}$ uniformly at random from $\{ \pm 1 \}^{\mathcal{B}_I}$. Since the resulting function $\tilde{h}_I = \sum_{K\in \mathcal{B}_I} \theta_Kh_K$ satisfies
  \begin{align*}
    \mathbb{E} \langle \tilde{h}_I, \widetilde{T} \tilde{h}_I \rangle = \sum_{K\in \mathcal{B}_I} \langle h_K, \widetilde{T}h_K \rangle
    \ge \sum_{K\in \mathcal{B}_I} \delta |K| = \delta |\mathcal{B}_I^{*}|,
  \end{align*}
  there is at least one realization $(\theta_K)_{K\in \mathcal{B}_I}$ such that
  \begin{equation}\label{eq:15}
    \langle \tilde{h}_I, \widetilde{T} \tilde{h}_I \rangle\ge \delta|\mathcal{B}_I^{*}|.
  \end{equation}
  We choose such a realization to define~$\tilde{h}_I$. If $k<N$, then the successors $\mathcal{B}_{I^{\pm }}$ are given by
  \begin{equation*}
    \mathcal{B}_{I^{\pm }} = \bigl\{ K\in \mathcal{G}_{s+k+1} : K\subset \{ \tilde{h}_I = \pm 1 \} \bigr\},
  \end{equation*}
  and the construction continues.

  Now let $A\colon Y_{\widetilde{N}}\to Y_N$ and $B\colon Y_N\to Y_{\widetilde{N}}$ be the operators associated with the resulting system $(\tilde{h}_I)_{I\in \mathcal{D}_{\le N}}$. We will use \Cref{pro:operators-A-B-almost-faithful} to estimate $\|A\|$ and~$\|B\|$. Put $\mu = |\mathcal{B}_{[0,1)}^{*}| = |\mathcal{G}_s^{*}|\ge 1/2$. Note that for $1\le k\le N$ and $I\in \mathcal{D}_k$, we have $|\mathcal{B}_I^{*}| \le |I|\mu$, and thus,
  \begin{equation*}
    \max_{I\in \mathcal{D}_k} (|I|\mu - |\mathcal{B}_I^{*}|)
    \le \sum_{I\in \mathcal{D}_k} (|I|\mu - |\mathcal{B}_I^{*}|)
    = |\mathcal{G}_s^{*}| - |\mathcal{G}_{s+k}^{*}|
    \le \frac{\mu}{(N/\varepsilon+1)2^N},
  \end{equation*}
  where the last inequality follows from~\eqref{eq:14}. Hence, we have
  \begin{equation*}
    |\mathcal{B}_I^{*}| \ge \Big( 1-\frac{1}{N/\varepsilon+1} \Big) |I|\mu,\qquad I\in \mathcal{D}_{\le N},
  \end{equation*}
  and this implies that
  \begin{equation*}
    0\le \frac{|I|}{|\mathcal{B}_I^{*}|} - \frac{1}{\mu} \le \frac{\varepsilon}{N\mu},\qquad I\in \mathcal{D}_{\le N}.
  \end{equation*}
  Thus, by \Cref{pro:operators-A-B-almost-faithful}, we have $\|B\|\le 1$ and $\|A\|\le 2(1+\varepsilon)$. Now put $T = A\widetilde{T}B$. Then by~\eqref{eq:15}, $T$ has $\delta$-large \emph{positive} diagonal.
\end{proof}

\begin{proof}[Proof of \Cref{cor:large-diagonal}]
By combining \Cref{pro:positive-reduction} with \Cref{thm:main-result}, we obtain the claimed result (note that the operator~$T$ obtained from \Cref{pro:positive-reduction} satisfies $\|T\|\le 2(1+\varepsilon)\|\widetilde{T}\|$).
\end{proof}

We conclude by posing the following question:
\begin{question}
  Does the conclusion of \Cref{cor:large-diagonal} still hold if the right-hand side of inequality~\eqref{eq:301} is replaced by a sub-exponential (e.g., polynomial) function of~$N$?
\end{question}


\noindent\textbf{Acknowledgments.}
The author would like to thank Richard Lechner and Paul~F.X. Müller for many helpful discussions.
This article is part of the author's PhD~thesis, which is being prepared at the Institute of
Analysis, Johannes Kepler University Linz.

\bibliographystyle{abbrv}%
\bibliographystyle{plain}%
\bibliography{bibliography}%

\end{document}